\documentclass[a4paper,11pt]{amsart}
\usepackage{amsthm,amsmath,amsfonts,amssymb}

\usepackage{color}

\newlength{\hchng}
\newlength{\vchng}
\newtheorem{thm}{Theorem}[section]
\newtheorem{prop}[thm]{Proposition}

\newtheorem{lemma}[thm]{Lemma}
\newtheorem{defn}[thm]{Definition}

\newtheorem{preremark}[thm]{Remark}
\newenvironment{remark}{\begin{preremark}\rm}{\medskip \end{preremark}}
\numberwithin{equation}{section}

\newcommand{\R}{\mathbb R}

\begin{document}
\date{}
\title[Long-time solvability  of  NSB equations]{Long-time solvability  of the Navier-Stokes-Boussinesq equations
with almost periodic initial large data}
\author{Slim Ibrahim  and Tsuyoshi Yoneda}
\email{ibrahim@math.uvic.ca} \urladdr{
http://www.math.uvic.ca/~ibrahim/}
\thanks{S. I. is partially supported by NSERC\# 371637-2009 grant and a start up fund from University
of Victoria}
\thanks{T. Y. is partially supported by PIMS Post-doc fellowship at the University
of Victoria, and partially supported by NSERC\# 371637-2009}
\maketitle
\begin{center}
Department of Mathematics and Statistics,
University of Victoria \\
PO Box 3060 STN CSC,
Victoria, BC, Canada, V8W 3R4 \\
\end{center}
\begin{center}
and
\end{center}
\begin{center}
Department of Mathematics, Hokkaido University\\ 
Sapporo 060-0810, Japan
\end{center}
\bibliographystyle{plain}

\vskip0.3cm \noindent {\bf Keywords:}
Navier-Stokes equation, Boussinesq approximation, almost periodic functions

\vskip0.3cm \noindent {\bf Mathematics Subject Classification:}
76D50,42B05

    \begin{abstract}
We investigate large time existence of solutions of the Navier-Stokes-Boussinesq equations with
 spatially almost periodic large  data when the density stratification is sufficiently large.
In 1996,  Kimura and Herring \cite{KH}  examined numerical simulations to show a stabilizing effect due to the stratification.
 They observed scattered two-dimensional pancake-shaped vortex patches lying almost in the horizontal plane.
Our result is a mathematical justification of the presence of such two-dimensional pancakes.
  To show the existence of solutions for large times,
we use $\ell^1$-norm of amplitudes.
 Existence for large times is then proven using techniques of fast singular oscillating limits and
bootstrapping argument from a global-in-time unique solution of the system of limit equations.
\end{abstract}


\section{Introduction}

Large-scale fluids such as atmosphere and ocean are parts of geophysical fluids, and the Coriolis force due to the earth rotation plays a significant role in the large scale flows considered in meteorology and geophysics.\\ 
Mathematically, it was first investigated by Poincar\'e \cite{Po}. Later on, the problem of strong Coriolis force was extensively studied. Babin, Mahalov and Nicolaenko (BMN) \cite{BMN1,BMN2} studied the incompressible rotating Navier-Stokes and Euler equations in the periodic case while Chemin, Desjardins, Gallagher and Grenier \cite{CDGG} analyzed the case of decaying data and more recently, the second author \cite{Y} considered the almost periodic case. Gallagher in \cite{Ga} studied a more abstract parabolic system. We also refer to Paicu \cite{P} for anisotropic viscous fluids, Benameur, Ibrahim and Majdoub \cite{BIM} for rotating Magneto-Hydro-Dynamic system and to Gallagher and Saint-Raymond \cite{GSR} for inhomogeheous rotating fluid equations.\\
Moreover on the one hand, the case when fluids are governed by both a strong Coriolis force and vertical stratification effects was investigated by BMN in \cite{BMN3} in the periodic setting and Charve in \cite{C} for decaying data. However, their studies do not cover the case when fluid equations are governed by the only effect of stratification.
It is known that a strong Coriolis force has a stabilizing effect (see \cite{BMN1}). However, in
BMN \cite[Section 9.2]{BMN4} the authors observed that for ideal fluids (i.e., with zero viscosity), the only effect of stratification leads to unbalanced dynamics. Moreover, the case of both strong Coriolis and stratification forces in the almost periodic setting seems to remain open. Finally, note that for the almost periodic case, energy type estimates cannot be used, and instead Fujita-Kato's approach has to be used.
On the other hand,
 Kimura and Herring \cite{KH}
 examined numerical simulations to show a stabilizing effect due to the effect of stratification for  viscous fluid.  They observed scattered two-dimensional pancake-shaped vortex patches lying almost in the horizontal plane.
Our result can be seen as a mathematical justification of the presence of such two-dimensional pancakes.
More precisely, we study long-time solvability for  Navier-Stokes-Boussinesq  equation
with stratification effects.
The Navier-Stokes-Boussinesq equations with stratification effects are governed by the following equations.
\begin{equation}\label{eq11}
\begin{cases}
\partial_t u -\nu\Delta u + (u \cdot  \nabla) u   +\nabla
p= g\rho e_3&x \in \mathbb{R}^3,\quad t>0\\
\partial_t \rho -\kappa\Delta \rho +(u \cdot \nabla) \rho  = -\mathcal{N}^2 u_3& x \in \mathbb{R}^3,\quad t>0\\
\nabla \cdot u=0&x\in \mathbb{R}^3,\quad t>0\\
u|_{t=0}=u_0 { \ },\quad \rho|_{t=0}=\rho_0
\end{cases}
\end{equation}
where the unknown functions $u = u(x, t) = (u_1,u_2,u_3)$, $\rho = \rho(x, t)$ and $p = p(x, t)$ are the fluid velocity,
 the thermal disturbances and the pressure, respectively. The parameters $\nu > 0$, $\kappa \geq 0$ and $g >0$ are the viscosity,
  the  thermal diffusivity and the gravity force, respectively. The parameter $\mathcal{N} > 0$ is  Brunt-V\"{a}is\"{a}l\"{a} frequency
  (stratification-parameter). Recall that  $\Delta :=(\partial_1^2+\partial_2^2+\partial_3^2)I_3$, $\nabla :=(\partial_1,\partial_2,\partial_3)$
  and $e_3:=(0,0,1)$.

Our method follows the ideas based on BMN. For the  limit
equations, we show that it is equivalent
to the 2D-Navier-Stokes equations\footnote{in the sense that there
is a one to one correspondence between solutions of the two
equations}, which is known to have, in the almost periodic setting, a unique global solution, see for example
\cite{GMY}. Then, we show that the global existence for the remainder equations in the limit equations.
 Since we handle not only periodic functions, we have to introduce a new analytic functional
setting, which is more suitable for the almost periodic situation as the second author did for the rotating fluid case in \cite{Y}. More
precisely, a straightforward application of an energy inequality is
impossible if the initial data is almost periodic. To overcome this
difficulty, we use $\ell^1$-norm of amplitudes with sum closed
frequency set. We recall the analytic functional  setting (see \cite{Y}) as follows:
\begin{defn} (Countable sum closed frequency set.)
\label{countable sum}
A countable set $\Lambda$ in $\mathbb{R}^3$ is called a sum closed frequency set if it satisfies the following properties:
$$
\Lambda=\{a+b:a,b\in\Lambda\}
\quad
 and
\quad -\Lambda=\Lambda.
$$
\end{defn}

\begin{remark}\
If  $\{e_j\}_{j=1}^3$ is the standard orthogonal basis in
$\mathbb{R}^3$, then the sets $\mathbb{Z}^3$, $\{m_1e_1+\sqrt
2m_2e_2 +m_3e_3:m_1,\cdots,m_3\in\mathbb{Z}\}$ and
$\{m_1e_1+m_2(e_1+e_2\sqrt 2)+m_3(e_2+e_3\sqrt
3):m_1,m_2,m_3\in\mathbb{Z}\}$ are examples of such countable sum
closed frequency sets. Clearly, the case $\mathbb Z^3$ corresponds to the periodic.
Each of the   other two cases are dense in $\mathbb{R}^3$ and therefore they correspond to ``purely" almost periodic setting.
\end{remark}
\begin{defn} (An $\ell^1$-type function space)
Let $BUC$ be the space of all bounded uniformly continuous functions defined in $\mathbb{R}^3$ equipped with the $L^\infty$-norm.
For a countable sum closed frequency set $\Lambda\subset\mathbb{R}^3$, let
$$
 X^\Lambda(\mathbb{R}^3):=
 \{u=\sum_{n\in\Lambda}\hat u_ne^{in\cdot x}\in BUC(\mathbb{R}^3): u_{-n}=u^*_n \quad for\quad n\in\Lambda,\
  \|u\|:=\sum_{n\in\Lambda}|u_n|<\infty\},
$$
where $u_n^*$ is the complex  conjugate coefficient of $u_n$.
\end{defn}

The second condition in Definition \ref{countable sum} is needed
to include real-valued almost periodic functions in $X^\Lambda$.

\begin{remark}
Note that functions in $\ell^1$ do not necessarily decay as $x\rightarrow\infty$. Also, this almost periodic setting is in general,
  different from the periodic case since the frequency set may have accumulation points. The almost periodic setting is somehow between the periodic and the full non-decaying cases.  
\end{remark}
Now, we define anisotropic dilation of the frequency set as follows.
\begin{defn}
For $\gamma=(\gamma_1,\gamma_2)\in (0,\infty)^2$, let
\begin{equation}\label{restriction}
\Lambda(\gamma):=\{(\gamma_1 n_1,\gamma_2 n_2,n_3)\in\mathbb{R}^3:(n_1,n_2, n_3)\in\Lambda\}.
\end{equation}
\end{defn}

Now, we specify  the following Quasi-Geostrophic equation (a part of limiting system),
 and assume (for the moment) that it has a scalar global solution $\theta =\theta (t)=\theta(t,x_1,x_2,x_3)$,
 \begin{equation}\label{QG}
\begin{cases}
\partial_t\theta -\Delta_3\theta+(-\Delta_h)^{-1/2}\left[(w\cdot \nabla)\left((-\Delta_h)^{1/2}\theta)\right)  \right]=0,\\
w=\left(-\partial_{x_2}(-\Delta_h)^{-1/2}\theta,\partial_{x_1}(-\Delta_h)^{-1/2}\theta\right)\\
\theta(t)|_{t=0}=\theta_0,
\end{cases}
\end{equation}
where $\Delta_h:=\partial_{x_1}^2+\partial_{x_2}^2$ and
$\Delta_3:=\partial_{x_1}^2+\partial_{x_2}^2+\partial_{x_3}^2$.
We will show that   the initial value problem
 for the QG equation admits a global-in-time unique solution in
  $C([0,\infty): X^\Lambda)$ with the  initial data $\theta_0=-\partial_{x_2}(-\Delta_h)^{-1/2}u_{0,1}+\partial_{x_1}(-\Delta_h)^{-1/2}u_{0,2}$.

\vspace{0.3cm}
More precisely,  we give an explicit  one-to-one correspondence between the QG and a 2D type Navier-Stokes equations
(for the existence of the unique global solution to 2D-Navier-Stokes equation with almost periodic initial data, see \cite{GMY}).
Now we state our main result.
\begin{thm}\label{main}
Let $\Lambda$ be a sum closed frequency set. There exists a set of frequencies dilation factors
 $\Gamma(\Lambda)\subset (0,\infty)^2$ such that\footnote{the complement set $\Gamma^c$  is at most countable.}:\\
for any $\gamma\in\Gamma$,
for any zero-mean value and divergence free initial vector field $u_0\in X^{\Lambda(\gamma)}$,
initial thermal disturbance $\rho_0\in X^{\Lambda(\gamma)}$, 
$\nu>0$, $\kappa\geq 0$  
and $T>0$, there exists
$N_0>g$ depending only on $\nu$, $\kappa$, $u_0$, $\rho_0$ such that if
$|N|>N_0$, then there exists a mild solution to the
equation (\ref{eq11}), $u(t) \in
C([0,T]: X^{\Lambda(\gamma)})$ with zero-mean value and  divergence free, and $\rho(t) \in
C([0,T]: X^{\Lambda(\gamma)})$ .
\end{thm}

\begin{remark}
For the periodic case, we do not need to restrict the frequency set to $\Gamma$ i.e. we can take $\Gamma(\Lambda)=(0,\infty)^2$.
 However, the computation in this case is more complicated and needs a``restricted convolution"
 type result in the spirit of  \cite{BMN2}.
\end{remark}

\section{Preliminaries}
Before going any further, we first recall the following few facts about the space $X^\Lambda$:

\begin{itemize}
\item $(X^\Lambda, \|\cdot\|)$ is a Banach space, and any almost periodic function $u\in X^\Lambda$ can be decomposed $u(x)=\Sigma_{n\in\Lambda}\hat u_ne^{inx}$,
 where each ``Fourier coefficient" $\hat u_n$ is uniquely determined by
$$
\hat u_n=\lim_{|B|\to\infty}\frac1{|B|}\int_{ B}u(x)e^{ix\cdot n}\;dx,
$$
and ${B}$ stands for a ball in $\R^3$ (see for example \cite{Co}).
\item $X^\Lambda$ is  closed subspace of $FM$, the Fourier preimage of the space of all finite Radon measures
  proposed by
Giga, Inui, Mahalov and Matsui in 2005 (see \cite{GIMM2,GIMS2,GJMY}).
\item Leray projection on almost periodic functions $\bar P=\{\bar P_{jk}\}_{j,k=1,2,3}$ is defined as
\begin{equation*}
\bar P_{jk}:=
\delta_{jk}+R_j R_k\quad (1\leq j,k \leq3)
\end{equation*}
with $\delta_{jk}$ is Kronecker's delta and $R_j$ is  the Riesz transform defined by
\begin{equation*}
R_j=\frac{\partial}{\partial x_j}(-\Delta)^{-\frac{1}{2}} { \ } for { \ } j=1,2,3.
\end{equation*}
The symbol $\sigma(R_j)$ of $R_j$ is $in_j/|n|$, where $i=\sqrt{-1}$ (see \cite{Co}).  Let $P$
 be the extended Leray projection with Fourier-multiplier $P_n=\{P_{n,ij}\}_{i,j=1,2,3,4}$ given by
\begin{equation*}
P_{n,ij}:=
\begin{cases}
\delta_{ij}-\frac{n_in_j}{|n|^2}\quad (1\leq i,j \leq3),\\
\delta_{ij}\quad (otherwise).
\end{cases}
\end{equation*}
\item Helmotz-Leray decomposition is defined on almost periodic functions in the same way as in the periodic case.
Namely, $u$ is uniquely decomposed as
\begin{equation*}
  u=w+\nabla \pi,
\end{equation*}
where $\pi=-(-\Delta)^{-1}\text{div}\ u$ and $w=\bar Pu$.
\end{itemize}

Now we rewrite the system ($\ref{eq11}$) in a more abstract way. Set $N:= \mathcal{N} \sqrt{g}$
 and $v \equiv (v_1,v_2,v_3,v_4):=(u_1,u_2,u_3,\frac{\sqrt{g}}{\mathcal{N}}\rho)$.
 Then $v$ solves 
\begin{equation}\label{eq31}
\begin{cases}
\partial_t v  -\tilde\nu\Delta v + NJ v + \nabla_3 p=-(v\cdot\nabla_3) v\\
v|_{t=0}=v_0\\
\nabla_3 \cdot v = 0,\\
 \end{cases}
\end{equation}
with $\tilde \nu=\hbox{diag}(\nu,\nu,\nu,\kappa)$, the initial data $v_0=(u_{0,1},u_{0,2},u_{0,3}$,
$\frac{ \sqrt{g}}{\mathcal{N}}\rho_0)$,
$\nabla_3:=(\partial_1,\partial_2,\partial_3,0)$,
\begin{equation*}
 J:=
\begin{pmatrix}
0      & 0 & 0 &0\\
0 & 0       & 0 &0\\
0      & 0       & 0 &-1\\
0      & 0       & 1 &0\\
\end{pmatrix},
\end{equation*}
and  $(v\cdot \nabla_3) =(v_1 \partial_1 + v_2 \partial_2 + v_3 \partial_3)$.

Observe that under the condition $N>g$ we have $\mathcal N>\sqrt g$
and therefore  $\|v_{0,4}\|=\|\frac{\sqrt g}{\mathcal
N}\rho_0\|<\|\rho_0\|$. We will assume this condition throughout the
paper.

Applying the extended Leray projection ${P}$ to  ($\ref{eq31}$), we obtain
\begin{equation}\label{eq32}
\begin{cases}
dv/dt +(-\tilde \nu\Delta+NS)v =- {P}( v \cdot \nabla_3 ) v
,\\
v|_{t=0}=P v_0=v_0,\\
\end{cases}
\end{equation}
with $S:= PJP$.
Recall that for $|n|_{h}\neq0$, the matrix $S_n:=P_nJP_n$ has the
following Craya-Herring orthonormal eigen basis
$\{q^1_n,q^{-1}_n,q^0_n ,q^{div}_n\}$ (see \cite{BMN3,EM})
associated to the eigenvalues $\{i\omega_n,-i\omega_n,0,0\}$ with
\begin{equation*}
\omega_n =\frac{|n|_h}{|n|},\   |n|_h=\sqrt{n_1^2+n_2^2}
\end{equation*}
and
\begin{align*}
q^1_n:=(q^1_{1,n}, q^1_{2,n}, q^1_{3,n}, q^1_{4,n} ): = & \frac{1}{\sqrt 2|n|_h^2}
(i\omega_nn_1n_3,i\omega_nn_2n_3,-i|n|_h^2\omega_n,|n|^2_h)=q^{-1*}_n\\
q^{-1}_n:=(q^{-1}_{1,n}, q^{-1}_{2,n}, q^{-1}_{3,n}, q^{-1}_{4,n} ):= &  \frac{1}{\sqrt 2|n|^2_h}
(-i\omega_nn_1n_3,-i\omega_nn_2n_3,i|n|_h^2\omega_n,|n|^2_h)
=q^{1*}_n\\
q^0_n := (q^0_{1,n}, q^0_{2,n}, q^0_{3,n}, q^0_{4,n} ):=& \frac{1}{|n|_h}(-n_2,n_1,0,0)=q_n^{0*}\\
q^{div}_n := (q^{div}_{1,n}, q^{div}_{2,n}, q^{div}_{3,n}, q^{div}_{4,n} ):=& \frac{1}{|n|}(n_1,n_2,n_3,0),
\end{align*}
where $q^{1*}_n=(q^1_n)^*$ is the conjugate of $q^1_n$.
The case when $|n|_h=0$ and $n_3\not=0$, we define
\begin{align*}
q^1_n: = & (1/2,1/2,0,1/\sqrt 2)\\
q^{-1}_n:= &  (-1/2,-1/2,0,1/\sqrt 2)\\
q^0_n :=& (-1/\sqrt 2,1/\sqrt 2,0,0 )\\
q^{div}_n := & (0,0,1,0).
\end{align*}
In fact, for $|n|_h=0$ and $n_3\not=0$, we have $S_n=P_nJP_n=0$. However, the above choice of the basis is uniquely determined by the conditions
$(\tilde \nu q^1_n\cdot q^{1*}_n)=\left(\frac{\nu+\kappa}{2}\right)$,  $(\tilde \nu q^{-1}_n\cdot q^{-1*}_n)=\left(\frac{\nu+\kappa}{2}\right)$
and $(\tilde \nu q^0_n\cdot q^{0*}_n)=\nu$ for \eqref{c}. Moreover, the divergence-free condition requires that $(\hat v_n(t)\cdot q^{div}_n)=0$,
 giving $q^{div}_n := (0,0,1,0)$.


Using Craya-Herring basis, one obtains an explicit representation of the solution to the linear version of ($\ref{eq32}$).
 For $n\in\Lambda$ and $\hat v_n:=(\hat v_{n,1},\hat v_{n,2}, \hat v_{n,3},\hat v_{n,4})$
 such that $\hat v_n\cdot \vec n=0$, we have
\begin{equation*}
e^{tNS_n}\hat v_n
=\sum_{\sigma_0\in\{-1,0,1\}}a_n^{\sigma_0}e^{tNS_n}q_n^{\sigma_0} =\sum_{\sigma_0\in\{-1,0,1\}}a_n^{\sigma_0} e^{i\sigma_0\omega_nNt}q_n^{\sigma_0},
\end{equation*}
with  $\vec n:=(n,0)=(n_1,n_2,n_3,0)$,
\begin{equation*}
\hat v_n
 =\sum_{\sigma_0\in\{-1,0,1\}}a_n^{\sigma_0} q_n^{\sigma_0}\quad\hbox{and}\quad a^{\sigma_0}_n:=(\hat v_n \cdot q^{\sigma_0*}_n).
\end{equation*}
Similarly, write a solution $v$ of \eqref{eq32} as
\begin{equation*}
v(t,x)=\sum_{n\in\Lambda}\hat v_n(t)e^{in\cdot x}.
\end{equation*}
From (\ref{eq32}), we derive for $n \in\Lambda$,
\begin{equation}\label{vequation}
\partial_t \hat v_n(t)=-\tilde \nu|n|^2\hat v_n(t)-S_n \hat v_n(t)-i
 P_n\sum_{n=k+m}(\hat v_k(t)\cdot \vec m)\hat v_m(t)\quad\text{with}\quad (\vec n\cdot \hat v_n(t))=0.
\end{equation}
In the sequel, we do not distinguish between $\vec n$ and $n$ unless a confusion occurs.
For $n\in \Lambda$  we have
\begin{equation*}
e^{tNS_n}\hat v_n(t)
=\sum_{\sigma_0\in\{-1,0,1\}}a_n^{\sigma_0}(t)e^{tNS_n}q_n^{\sigma_0} =\sum_{\sigma_0\in\{-1,0,1\}}a_n^{\sigma_0}(t) e^{i\sigma_0\omega_nNt}q_n^{\sigma_0},
\end{equation*}
where $a^{\sigma_0}_n(t):=(\hat v_n(t)\cdot q^{\sigma_0*}_n)$. From equation \eqref{vequation}, we get for $\sigma_0=-1,0,1$,
\begin{eqnarray}
\nonumber
\partial_t a^{\sigma_0}_n(t)&=&-a^{\sigma_0}_n(t)((|n|^2 \tilde\nu+NS_n)q^{\sigma_0}_n\cdot q^{\sigma_0*}_n)\\
& -&i\sum_{n=k+m,\;\sigma_1,\sigma_2\in \{-1,0,1\}}c^{\sigma_1}_kc^{\sigma_2}_m
(q^{\sigma_1}_k\cdot m)(P_n q^{\sigma_2}_m\cdot q^{\sigma_0*}_n).
\nonumber
\end{eqnarray}
Note that $P_n$ is self adjoint and $P_nq^{\sigma_0*}_n=q^{\sigma_0*}_n$.
Setting $c^{\sigma_0}_n(t):=e^{-itN \sigma_0\omega_n}a^{\sigma_0}_n(t)$ leads to the following equation
\begin{eqnarray}
\nonumber
\partial_t c^{\sigma_0}_n(t)&=&-c^{\sigma_0}_n(t)|n|^2 (\tilde\nu q^{\sigma_0}_n\cdot q^{\sigma_0*}_n)\\
& -&i\sum_{n=k+m,\;\sigma_1,\sigma_2\in \{-1,0,1\}}e^{iNt\omega^\sigma_{nkm}}c^{\sigma_1}_kc^{\sigma_2}_m
(q^{\sigma_1}_k\cdot m)(q^{\sigma_2}_m\cdot q^{\sigma_0*}_n).
\nonumber
\end{eqnarray}
where, $\omega^\sigma_{nkm}:=(-\sigma_0\omega_n+\sigma_1\omega_k+\sigma_2\omega_m )$. Now we split the nonlinear part into the ``resonant"
 (independent of $N$) and non ``resonant" two parts
defined by
 $$
 \bar B_n^{\sigma_0}(g^{\sigma_1},h^{\sigma_2}):=-i\sum_{n=k+m,\;\omega^\sigma_{nkm}=0}
 (q^{\sigma_1}_k\cdot m)(q^{\sigma_2}_m\cdot q^{\sigma_0*}_n)
 g^{\sigma_1}_kh^{\sigma_2}_m
 $$
 and
 $$
 \tilde B_n^{\sigma_0}(Nt, g^{\sigma_1},h^{\sigma_2}):=-i\sum_{n=k+m,\;\omega^\sigma_{nkm}\not=0}
 (q^{\sigma_1}_k\cdot m)(q^{\sigma_2}_m\cdot q^{\sigma_0*}_n)
g^{\sigma_1}_kh^{\sigma_2}_m
 \exp (i\omega^\sigma_{nkm} Nt),
 $$
respectively.
In addition, observe that  we have the following estimates:

\begin{equation}\label{desiredestimates}
\begin{cases}
\|e^{-\nu|n|^2t}\tilde B^{\sigma_0}_n(Nt, g^{\sigma_1},h^{\sigma_2})\|\leq \frac{C_{\nu}}{t^{1/2}}\|g^{\sigma_1}\|\|h^{\sigma_2}\|\\
 \|e^{-\nu|n|^2t}\bar B^{\sigma_0}_n(g^{\sigma_1},h^{\sigma_2})\|\leq \frac{C_{\nu}}{t^{1/2}}\|g^{\sigma_1}\|\|h^{\sigma_2}\|
 \end{cases}
\end{equation}
(for $\sigma_0=-1,0,1$) obtained by estimating the first derivative of the heat kernel as follows
\begin{equation*}
\sup_{n\in\Lambda}\left | |n|e^{- \nu |n|^2 t}\right|\leq \frac{C_{\nu}}{t^{1/2}}.
\end{equation*}
The constant  $C_\nu>0$  is independent of $N$.

Then we have the following equations:
\begin{equation}\label{original}
\begin{cases}
\partial_t c_n^0(t)=-\nu |n|^2c_n^0(t)+\sum_{(\sigma_1,\sigma_2)\in\{-1,0,1\}^2}\left(\bar B^0_n(c^{\sigma_1},c^{\sigma_2})
+\tilde B^0_n(Nt, c^{\sigma_1},c^{\sigma_2})\right)\\ \\
\partial_t c^{\sigma_0}_n(t)=-\left(\frac{\nu+\kappa}{2} \right)|n|^2c^{\sigma_0}_n(t)
+\sum_{(\sigma_1,\sigma_2)\in \{-1,0,1\}^2}\left(\bar B^{\sigma_0}_n(c^{\sigma_1},c^{\sigma_2})
+\tilde B^{\sigma_0}_n(Nt, c^{\sigma_1},c^{\sigma_2})    \right),
\end{cases}
\end{equation}
for $\sigma_0=\pm1$.
From the condition $\omega^{\sigma}_{nkm}=0$, we easily see that the terms $\bar B_n^{0}(c^{1},c^{1})$, $\bar B_n^{0}(c^{-1},c^{-1})$,
 $\bar B_n^{0}(c^{0},c^{\pm 1})$, $\bar B_n^{0}(c^{\pm 1},c^{0})$,  $\bar B_n^{\pm 1}(c^{\mp 1},c^{0})$, $\bar B_n^{\pm 1}(c^{0},c^{\mp 1})$
  and $\bar B_n^{\pm 1}(c^{0},c^{0})$  disappear.
Now, we define the ``limit equations" by
\begin{equation}\label{c}
\begin{cases}
\partial_t c_n^0(t)=-\nu |n|^2c_n^0(t)+\bar B^0_n(c^0,c^0)+\bar B^0_n(c^1,c^{-1})+\bar B^0_n(c^{-1},c^1),\\ \\
\partial_t c^{\sigma_0}_n(t)=-\left(\frac{\nu+\kappa}{2} \right)|n|^2c^{\sigma_0}_n(t)
+\sum_{(\sigma_1,\sigma_2)\in \{-1,0,1\}^2\setminus D}\bar B^{\sigma_0}_n(c^{\sigma_1},c^{\sigma_2}), \quad \sigma_0=\pm1,
\end{cases}
\end{equation}
where $D:=\{(0,0), (-1, 0), (0,-1)\}$ for $\sigma_0=1$ and  $D:=\{(0,0), (1, 0), (0,1)\}$ for $\sigma_0=-1$.
Formally, we can get \eqref{c} from \eqref{original} when $N\to \infty$. We will justify this convergence in Lemma \ref{fas}.
Now we show that there is more non trivial cancellation in the limit equations. More precisely,
\begin{lemma}
We have
\begin{equation*}
\bar B^0_n(c^1,c^{-1})+\bar B^0_n(c^{-1},c^1)=0.
\end{equation*}
\end{lemma}
\begin{proof}
To prove the lemma, it suffices to show
\begin{equation}\label{cancellation}
(q^1_k\cdot m)(q^{-1}_m\cdot q^{0*}_n)+(q^{-1}_m\cdot k)(q^1_k\cdot q^{0*}_n)=0\quad\text{for any}\quad n=k+m\quad \text{with}\quad \omega_k=\omega_m.
\end{equation}
First we show that $\omega_k=\omega_m$ if and only if
\begin{equation*}
k,m\in\{n\in\mathbb{Z}^3: |n|_h^2=\lambda n_3^2\}\quad \text{for some}\quad \lambda>0.
\end{equation*}

($\Leftarrow$): This direction is clear. Thus we omit it.

($\Rightarrow$): Rewrite the identity  $\omega_k=\omega_m$ as $F(X)=F(Y)$, where $X:=|k|_h^2/k^2_3$, $Y:=|m|_h^2/ m_2^3$  and
$F(X):=X/(X+1)$. Since the function $F$ is monotone increasing, we see $X=Y$. This means that
\begin{equation*}
k_3=\pm\frac{|k|_h}{\sqrt \lambda}\quad\text{and}\quad m_3=\pm\frac{|m|_h}{\sqrt \lambda}.
\end{equation*}
We only consider the case  $k_3=\frac{|k|_h}{\sqrt \lambda}$ and $m_3=\frac{|m|_h}{\sqrt \lambda}$, since the other cases are similar.
A direct calculation shows that
\begin{eqnarray*}
(q^1_k\cdot m)(q^{-1}_m\cdot q^{0*}_n)&=&\frac{1}{\sqrt 2\lambda |m||k||n|_h}\left(\frac{k_h\cdot m_h}{\lambda}-\frac{m_3|k|}{\sqrt{1+\lambda^2}}\right)
(-m_2k_1+m_1k_2),\\
(q^{-1}_m\cdot k)(q^{1}_k\cdot q^{0*}_n)&=&\frac{1}{\sqrt 2\lambda |m||k||n|_h}\left(\frac{k_h\cdot m_h}{\lambda}-\frac{k_3|m|}{\sqrt{1+\lambda^2}}\right)
(-k_2m_1+k_1m_2).\\
\end{eqnarray*}
By $k_3=\frac{|k|_h}{\sqrt \lambda}$ and $m_3=\frac{|m|_h}{\sqrt \lambda}$, we have \eqref{cancellation}.
\end{proof}

Now we show that the function $c^0$ in the limit equations satisfies a quasi geostrophic (QG) equation type
 and that this QG equation is equivalent to the 2D type Navier-Stokes equation.
By the following lemma,  we can see that the function $c^0$ satisfies the QG equation \eqref{QG}.

\begin{lemma}\label{qg nonlinear}\
Let $|n|_h=\sqrt{n_1^2+n_2^2}$.
The resonant part
$ \bar B_n^0(c^{0},c^{0})$ can be expressed  as follows:

\begin{equation*}
 \bar B^{0}_n(c^0,c^0)=
-\sum_{n=k+m}
\frac{i(k\times m)|m|_h}{|k|_h|n|_h}
c_{k}^0 c_{m}^0.
\end{equation*}

\end{lemma}

\begin{proof}
Since
$q^0_n :=  \frac{1}{|n|_h}(-n_2,n_1,0,0)$
and $q^0_n=\frac{1}{|n|_h}(|k|_hq^0_k+|m|_hq^0_m)$ for $n=k+m$,
we have
\begin{eqnarray*}
\bar B_n^0(c^0,c^0)&=&-\sum_{n=k+m}c^0_kc^0_m  (q_k^0\cdot im)(q^0_m\cdot q^{0*}_n)\\
&=&-\sum_{n=k+m}\frac{i}{|k|_h|n|_h}(k_2m_1-k_1m_2)\left(q_m^0\cdot (|k|_hq_k^{0*}+|m|_hq_m^{0*})\right)c^0_kc^0_m\\
&=&\sum_{n=k+m}-\frac{i|m|_h}{|k|_h|n|_h}(k_2m_1-k_1m_2)c^0_kc^0_m\\
& &
-
\sum_{n=k+m}\frac{i}{|n|_h}(k_2m_1-k_1m_2)(q^0_m\cdot q_k^{0*})c^0_kc^0_m.
\end{eqnarray*}
Since $k\times m=-(m\times k)$,
we see that 
\begin{equation*}
\sum_{n=k+m}\frac{i}{|n|_h}(k_2m_1-k_1m_2)(q^0_m\cdot q_k^{0*})c^0_kc^0_m=0,
\end{equation*}
which leads to the desired formula.
\end{proof}
Now we show that there is a one-to-one correspondence between the QG and a 2D type Navier-Stokes equations.

\begin{lemma}
Let $\Delta_h=\partial_{x_1}^2+\partial_{x_2}^2$ and 
\begin{eqnarray*}
 w:=(w_1(x_1,x_2,x_3,t),w_2(x_1,x_2,x_3,t)):=
\left(\sum_{n\in\Lambda}\hat w_{1,n}(t)e^{in\cdot x},\sum_{n\in\Lambda}\hat w_{2,n}(t)e^{in\cdot x}\right)
\end{eqnarray*}
and define $\theta=\theta (t,x_1,x_2,x_3):=(-\Delta_h)^{-1/2}\text{rot}_2 w$, with $\text{rot}_2$ is the 2 dimensional curl given by
$$
\text{rot}_2w=\partial_2 w_1-\partial_1w_2.
$$
Then, $w$ solves the following 2D type Navier-Stokes equation

\begin{equation}\label{2DNS}
 \begin{cases}
 \partial_t w-\Delta w+(w\cdot\nabla_2)w+\nabla_2p=0,\\
 \nabla_2\cdot w=0,\ w|_{t=0}=w_0
 \end{cases}
 \end{equation}
 if and only if $\theta$ solves \eqref{QG}.
\end{lemma}

\begin{proof}
Recall that $-\Delta_h=-(\partial_{x_1}^2+\partial_{x_2}^2)$. First observe that for $\theta=(-\Delta_h)^{-1/2}\text{rot}_2 w=\sum_{n\in\Lambda}\hat \theta_n(t)e^{in\cdot x}$,  we have by $\nabla_2\cdot w=0$,

$$
w=(i\sum_{n\in\Lambda}\frac{n_2}{|n|_h}\hat \theta_n(t)e^{in\cdot x}, -i\sum_{n\in\Lambda}\frac{n_1}{|n|_h}\hat \theta_n(t)e^{in\cdot x})=
(\partial_2(-\Delta_h)\theta, -\partial_1(-\Delta_h)\theta).
$$
Then, apply rot$_2$ to (\ref{2DNS}), we get

\begin{equation}\label{vortexequation}
\partial_t\text{rot}_2 w-\Delta\text{rot}_2  w+(w\cdot \nabla_2)\text{rot}_2 w=0,
\end{equation}
here, we used the fact that
\begin{equation}\label{commute}
(w\cdot\nabla_2)\text{rot}_2w=\text{rot}_2[(w\cdot\nabla_2)w].
\end{equation}

Finally, apply $(-\Delta_h)^{-1/2}$ to both sides of \eqref{vortexequation}, we see that $\theta=(-\Delta_h)^{-1/2}\text{rot} w$ satisfies the desired QG equation \eqref{QG}.
 Conversely,  applying $L:=(-(-\Delta_h)^{-1}\partial_{x_2},(-\Delta_h)^{-1}\partial_{x_1})$ (which commutes with $\Delta$) to (\ref{vortexequation}), and by \eqref{commute} we can see that $L\text{rot}_2$ is nothing but the two dimensional Leray projection. Therefore, this implies (\ref{2DNS}) as desired.
\end{proof}

\begin{remark}
We refer to \cite{GMY} for the existence of the unique global solution to 2D type Navier-Stokes equation \eqref{2DNS} with almost periodic initial data.
\end{remark}

In what follows and in order to show  the main theorem, we need the following  lemma  (which is needed only for the almost periodic case)
 on the dilation of the frequency set \eqref{restriction}. This kind of restrictions is
technical. However we do not know whether or not  such constraints
are removable. This means that the general almost periodic setting seems to remain open.

\begin{lemma}\label{pnkm}
For $n,k,m \in \Lambda$ and   $\gamma=(\gamma_1,\gamma_2)\in(0,\infty)^2$, define $\tilde n=(\gamma_1 n_1,\gamma_2 n_2,
n_3)$, $\tilde k=(\gamma_1 k_1,\gamma _2 k_2,k_3)$ and $\tilde m=(\gamma_1 m_1,\gamma_2 m_2,
m_3)$.
Let
$$
  P_{ nk m}(\gamma):=|  \tilde n|^8| \tilde k|^8|\tilde  m|^8\prod_{\sigma\in\{-1,1\}^3}\omega^\sigma_{\tilde n\tilde k\tilde m}.
$$
Given a frequency set $\Lambda$, there is $\Gamma:=\Gamma(\Lambda)\subset(0,\infty)^2$ s.t. for any $\gamma\in\Gamma$, $P_{nkm}(\gamma)\not=0$ for any $n$, $k$, $m\in\Lambda$ such that $(n_h,k_h,m_h)\not=(0,0,0)$.
\end{lemma}
\begin{remark}
If  $n_h$, $k_h$, $m_h=0$, then $\bar B^{\pm 1}_n(c^{\pm 1},c^{\pm 1})=0$.
\end{remark}

\begin{proof}
Define $\Gamma$ by
$$
\Gamma:=\{\gamma\in(0,\infty)^2:
 P_{n k m}(\gamma)\not=0\quad
 for\quad all
 \quad n, k, m\in\Lambda \quad with \quad |n|_h,|k|_h, |m|_h\not=0\}.
$$
Nothe that  if
$\gamma\in\Gamma$, then $\omega^\sigma_{\tilde n\tilde
k\tilde m}\not=0$.
We show that   $\Gamma$ cannot be  empty.
 By a direct calculation, we have 
\begin{eqnarray*}
P_{nkm}(\gamma)&=&
|\tilde n|^8|\tilde k|^8|\tilde m|^8\\
& &
\left( (\omega_{\tilde n}+\omega_{\tilde k}+\omega_{\tilde m})
(-\omega_{\tilde n}+\omega_{\tilde k}+\omega_{\tilde m})
 (\omega_{\tilde n}-\omega_{\tilde k}+\omega_{\tilde m})
(\omega_{\tilde n}+\omega_{\tilde k}-\omega_{\tilde m})  \right)^2\\
&=&
|\tilde n|^8|\tilde k|^8|\tilde m|^8
\left( \omega_{\tilde n}^2-(\omega_{\tilde k}+\omega_{\tilde m})^2\right)^2
\left ((\omega_{\tilde n}^2-(\omega_{\tilde k}-\omega_{\tilde m})^2\right)^2\\
&=&
|\tilde n|^8|\tilde k|^8|\tilde m|^8
\left( (\omega_{\tilde n}^2-\omega_{\tilde k}^2-\omega_{\tilde m}^2)^2-4
\omega_{\tilde k}^2\omega_{\tilde m}^2\right)^2\\
&=&
|\tilde n|^8|\tilde k|^8|\tilde m|^8
\left( \omega_{\tilde n}^4+\omega_{\tilde k}^4+\omega_{\tilde m}^4
-2\omega_{\tilde k}^2\omega_{\tilde m}^2
-2\omega_{\tilde m}^2\omega_{\tilde n}^2
-2\omega_{\tilde n}^2\omega_{\tilde k}^2
\right)^2\\
&=&
|\tilde n|_h^4|\tilde k|^4|\tilde m|^4+|\tilde n|^4|\tilde k|_h^4|\tilde m|^4
+|\tilde n|^4|\tilde k|^4|\tilde m|^4_h\\
& &
-2|\tilde n|^2|\tilde n|^2_h|\tilde k|^2|\tilde k|^2_h|\tilde m|^4
-2|\tilde n|^2|\tilde n|^2_h|\tilde k|^4|\tilde m|^2|\tilde m|_h^2
-2|\tilde n|^4|\tilde k|^2|\tilde k|^2_h|\tilde m|^2|\tilde m|_h^2\\
&=&
-3n_1^2k_1^2m_1^2 \gamma_1^6-3n_2^2k_2^2m_2^2 \gamma_2^6\\
& &
-3n_1^2k_1^2m_2^2 \gamma_1^4\gamma_2^2
-3n_1^2k_2^2m_1^2 \gamma_1^4\gamma_2^2
-3n_2^2k_1^2m_1^2 \gamma_1^4\gamma_2^2\\
& &
-3n_2^2k_2^2m_1^2 \gamma_1^2\gamma_2^4
-3n_2^2k_1^2m_2^2 \gamma_1^2\gamma_2^4
-3n_1^2k_2^2m_2^2 \gamma_1^2\gamma_2^4
+\cdots.
\end{eqnarray*}
Since  $|n|_h, |k|_h, |m|_h\not=0$, then the highest order terms
never disappear. This means that
$$
|\{\gamma: \cup_{n,k,m}P_{nkm}(\gamma)=0\}|=0.
$$
Thus the complement set of $\Gamma$ is countable (which means that $\Gamma$ is a non-empty set).
\end{proof}

Now we show that  the limit equations have a global solution. In the almost periodic case,
the non-resonant part $\bar B^{\pm 1}(c^{\pm 1},c^{\pm 1})$ disappears just by restricting the frequencies set to $\Lambda(\gamma)$.
However, the periodic case is more subtle as we need a lemma on restricted convolution (see \cite{BMN2}).

\begin{lemma}\label{glo}
Let $\Lambda$ be a sum closed frequency set. If $\Lambda=\mathbb{Z}^3$, take $\gamma\in(0,\infty)^2$, otherwise we restrict it to $\gamma\in\Gamma(\Lambda)$.
Then for $\sigma_0=\pm 1$ there exists a global-in-time unique solution $c^{\sigma_0}(t)$ to equations (\ref{c})
   such that
 $c^{\sigma_0}(t)\in C([0,\infty):\ell^1(\Lambda(\gamma)))$ with $(c^{\sigma_0}_n(t)\cdot n)=0$ for all $n\in\Lambda(\gamma)$ and $c^{\sigma_0}_0(t)=0$.
\end{lemma}

\begin{proof}
Recall that $\tilde n=(\gamma_1 n_1,\gamma_2 n_2,
n_3)$, $\tilde k=(\gamma_1 k_1,\gamma _2 k_2,k_3)$ and $\tilde m=(\gamma_1 m_1,\gamma_2 m_2,
m_3)$.
First we consider  the almost periodic case. By restricting $\gamma\in\Gamma$,  we can eliminate the worst non-linear
term using Lemma \ref{pnkm}. More precisely, for all $ n\in \Lambda$ and $\sigma=(\sigma_1,\sigma_2,\sigma_3)\in\{-1,1\}^3$, the term
$$
\bar B^{\sigma_0}_{\tilde n}(c^{\sigma_1},c^{\sigma_2})= \sum_{\stackrel{\tilde n=\tilde k+\tilde m}{\omega^\sigma_{\tilde n\tilde k\tilde m}=0}}
(q^{\sigma_1}_{\tilde k}\cdot i\tilde m)(q^{\sigma_2}_{\tilde m}\cdot q^{\sigma_0*}_{\tilde n})c_{\tilde k}^{\sigma_1}c_{\tilde m}^{\sigma_2}\quad\text{disappears}.
$$
Then we have two coupled  linear equations for $\{c^{-1}_n\}_n$ and $\{c^1_n\}_n$.
In this case, the global existence will immediately follow from estimates  \eqref{desiredestimates}.\\
However, the periodic case requires more details. For $\alpha>0$ and  $p\geq 1$, define the weighted $\ell_p^\alpha$ norm as
\begin{equation*}
\|c\|_{\ell^\alpha_p}:=\left(\sum_{n\in\Lambda}|n|^{p\alpha}|c_n|^p\right)^{1/p}.
\end{equation*}
The main step is to show an {\it \`a-priori} bound on $c^{\pm 1}$ in $\ell^1_2$.  Since the 3D type Navier-Stokes equation \eqref{c}
 is  subcritical in the space  $\ell^s_2$ with $s>1/2$, then
 a bootstrap argument (using the dissipation, see \cite[Proposition 15.1]{Le} for example) enables us to conclude that
\begin{equation*}
c^{\pm 1}(t)\in L^\infty_{loc}([0,\infty):\ell^1_2)\cap L^\infty_{loc}((0,\infty):\ell^s_2)\quad\text{for}\quad 1<s<2,
\end{equation*}
whenever the initial data $c^{\pm 1}(0)\in\ell^1_2$.
More precisely, by \eqref{c} we write the following mild formulation:
\begin{equation}\label{limit mild solution}
c_n^{\sigma_0}(t)=e^{\frac{\nu+\kappa}{2}|n|^2t}c_n^{\sigma_0}(t_0)-\sum_{(\sigma_1,\sigma_2)\in\{-1,0,1\}^2\setminus D}\int_{t_0}^t
e^{-\frac{\nu+\kappa}{2}(t-\tau+t_0)|n|^2}\bar B^{\sigma_0}_n(c^{\sigma_1},c^{\sigma_2})d\tau
\end{equation}
for $0<t_0<t$.
We have a good estimate for the  heat kernel with fractional Laplacian,
\begin{equation*}
|n|^se^{-|n|^2t}=t^{-s/2}|t^{1/2}n|^se^{-|t^{1/2}n|^2}\leq
t^{-s/2}\frac{C_\epsilon}{1+|t^{1/2}n|^{1/\epsilon}}\leq t^{-s/2-\epsilon/2}\frac{C_\epsilon}{1+|n|^{1/\epsilon}}
\end{equation*}
for $0<t\leq 1$ and  $\epsilon>0$.
Then by  H\" older's and Young's  inequality for the discrete case, we have the following estimate from \eqref{limit mild solution}:
\begin{eqnarray}\label{regularity estimate}
\|c^{\sigma_0}(t)\|_{\ell^s_2}&\leq &C_{\kappa,\nu,s}\bigg[\|c^{\sigma_0}(t_0)\|_{\ell^1_2}
+
\int_{t_0}^t(t-\tau+t_0)^{-s/2-\epsilon/2}\times\\
& &\sum_{(\sigma_1,\sigma_2)\in\{-1,0,1\}^2}\left(\|c^{\sigma_1}(\tau)\|_{\ell^1_2}\|c^{\sigma_2}(\tau)\|_{\ell^0_2}+
\|c^{\sigma_1}(\tau)\|_{\ell^0_2}\|c^{\sigma_2}(\tau)\|_{\ell^1_2}\right)d\tau\bigg].
\nonumber
\end{eqnarray}
If $1<s<2$, then $\|c^{\sigma_0}(t)\|_{\ell^s_2}$ is finite since the right hand side of  \eqref{regularity estimate} is finite.
This means that  $c^{\sigma_0}\in L^\infty_{loc}((0,\infty):\ell^s_2)$ for $1<s<2$.
Then, thanks to Bernstein's lemma (which can be applied only for the periodic case),  $\sum_n|c_n|\leq \|c\|_{\ell^\alpha_2}$ for $\alpha>3/2$,
we get an {\it \`a-priori} bound of the  $\ell^1$-norm.
 Now we show an  {\it \`a-priori} bound of $c^{\pm 1}$ in $\ell^1_2$. Multiply equation (\ref{c}) by $|n|^2(c_n^{\sigma_0})^*$ and summing, we obtain
 \begin{eqnarray*}
\sum_{{\sigma_0}\in\{-1,1\}}\sum_{n\in\mathbb Z^3} \bigg[\frac12\partial_t(|n| c_n^{\sigma_0}(t))^2&+&\left(\frac{\nu+\kappa}2\right)( |n|^2 c_n^{\sigma_0}(t))^2\bigg]\\
&=&\sum_{n\in\mathbb Z^3} \sum_{{\sigma_0}\in\{-1,1\}}   \sum_{(\sigma_1,\sigma_2)\in\{-1.0,1\}^2\setminus D}{\bar B}_n^{\sigma_0}(c^{\sigma_1},c^{\sigma_2})|n|^2c_n^{\sigma_0*}(t).
\nonumber
 \end{eqnarray*}
 Notice  that
 \begin{eqnarray*}
 (c^{\sigma_0}_n)^*&=&
 \left[(\hat v_n\cdot q^{\sigma_0*}_n)e^{-itN\sigma_0\omega_n}\right]^*
 =
 (\hat v^*_n\cdot q^{\sigma_0}_n)e^{itN\sigma_0\omega_n}\\
 &=&
  (\hat v_{-n}\cdot (q^{-\sigma_0}_{-n})^*)e^{-itN(-\sigma_0)\omega_{-n}}= c_{-n}^{-\sigma_0}.
 \end{eqnarray*}

Moreover, from one hand, a direct calculation using H\"older's then young's inequalities shows that
\begin{multline}\label{aprioriestimate2}
\left|\sum_n\bar B^{\pm 1}_n(c^{0},c^{\pm 1})|n|^2c^{\mp 1}_{-n}\right|+
\left|\sum_n\bar B^{\pm 1}_n(c^{\pm 1},c^{0})|n|^2c^{\mp 1}_{-n}\right|\\
\leq
C\|c^{\mp 1}\|_{\ell^2_2}\left(\|c^{\pm 1}\|_{\ell^0_2}\|c^{0}\|_{\ell^1_1}+ \|c^{\pm 1}\|_{\ell^1_2}\|c^{0}\|_{\ell^0_1}   \right)
\leq \epsilon\|c^{\mp 1}\|_{\ell_2^2}^2+
C\left(\|c^{\pm 1}\|_{\ell^0_2}^2\|c^{0}\|_{\ell^1_1}^2+ \|c^{\pm 1}\|_{\ell^1_2}^2\|c^{0}\|_{\ell^0_1}^2   \right)
\end{multline}
for sufficiently small $\epsilon>0$.
On the other hand, we  show that

\begin{equation}\label{skew1}
\left|\sum_n\bar B^{\pm 1}_n(c^{\pm 1},c^{\pm 1})|n|^2c^{\mp 1}_{-n}\right|
\leq
\sum_n\sum_{-n=k+m}|m||c^{\pm 1}_m||k||c^{\pm 1 }_k||n||c^{\mp 1}_n|.
\end{equation}
Using  identities  $\alpha\times(\alpha\times \beta)=-|\alpha|^2\beta$ and  $(\alpha\cdot (\beta\times\gamma))=((\alpha\times\beta)\cdot\gamma)$,  we see that
\begin{eqnarray*}
\sum_n\bar B^{\pm 1}_n(c^{\pm 1},c^{\pm 1})|n|^2c^{\mp 1}_{-n}
&=&
\sum_n\sum_{-n=k+m}(im\cdot q^{\pm 1}_k)\left(q^{\pm 1}_m\cdot (in\times (in \times q^{\mp 1}_n))\right)c^{\pm 1}_kc^{\pm 1}_mc^{\mp 1}_{n}\\
&=&
\sum_n\sum_{-n=k+m}(im\cdot q^{\pm 1}_k)\left((q^{\pm 1}_m\times in)\cdot( in \times q^{\mp 1}_n)\right)c^{\pm 1}_kc^{\pm 1}_mc^{\mp 1}_{n}\\
&=&
-\sum_n\sum_{-n=k+m}(im\cdot q^{\pm 1}_k)(q^{\pm 1}_m\times ik)( in \times q^{\pm 1}_n)c^{\mp 1}_kc^{\pm 1}_mc^{\mp 1}_{n}\\
& &
-\sum_n\sum_{-n=k+m}(im\cdot q^{\pm 1}_k)(q^{\pm 1}_m\times im)( in \times q^{\pm 1}_n)c^{\mp 1}_kc^{\pm 1}_mc^{\mp 1}_{n}.
\end{eqnarray*}
By the skew-symmetry, namely,
\begin{eqnarray*}
\sum_n\sum_{-n=k+m}(im\cdot q^{\pm 1}_k)(q^{\pm 1}_m\times im)( in \times q^{\mp 1}_n)c^{\pm 1}_kc^{\pm 1}_mc^{\mp 1}_{n}&=&\\
-\sum_n\sum_{-n=k+m}(in\cdot q^{\pm 1}_k)(q^{\pm 1}_m\times im)( in \times q^{\mp 1}_n)c^{\pm 1}_kc^{\pm 1}_mc^{\mp 1}_{n}&=&\\
-\sum_m\sum_{-m=k+n}(in\cdot q^{\pm 1}_k)(q^{\pm 1}_m\times im)( in \times q^{\mp 1}_n)c^{\pm 1}_kc^{\pm 1}_mc^{\mp 1}_{n},& &\\
\end{eqnarray*}
 the second term disappears which obviously leads to (\ref{skew1}).
Now to have the {\it \`a priori} bound of $c^{\pm 1}$ in $\ell^1_2$, we apply a smoothing via time averaging effect due to \cite{BMN2}.

\begin{prop}\label{restconvprop}{Restricted convolution}\cite[Theorem 3.1 and Lemma 3.1]{BMN2}
Assume that  the following holds for $|n|_h\neq0$

\begin{equation}\label{restconv}
\sup_n\sum_{k:k+m+n=0,k\in \Sigma_i}\chi(n,k,m)|k|^{-1}\leq C2^{i}
\end{equation}
for every $i=1,2,\cdots$, where
\begin{equation*}
\Sigma_i:=\{k:2^i\leq |k|\leq 2^{i+1}\},
\quad
\chi(n,k,m)=
\begin{cases}
1\quad\text{if}\quad P_{nkm}(1)=0\\
0\quad\text{if}\quad P_{nkm}(1)\not=0.
\end{cases}
\end{equation*}
Then we have
\begin{eqnarray}\label{apriori0}
\left|\sum_n\bar B^{\pm 1}_n(c^{\pm 1},c^{\pm1})|n|^2c^{\mp 1}_{-n}\right|
&\leq&
C\|c^{\pm 1}\|_{\ell^2_2}\|c^{\pm 1}\|_{\ell^1_2}^2\\
\nonumber
&\leq& C\|c^{\pm 1}\|_{\ell_2^1}^4+\epsilon\|c^{\pm 1}\|_{\ell^2_2}^2.
\nonumber
\end{eqnarray}
\end{prop}
Now, combining (\ref{aprioriestimate2}) and (\ref{apriori0}), we obtain the following estimate on $c^{\pm 1}$ in $\ell^1_2$, namely,
\begin{equation}\label{apply Gronwall}
\|c^{\pm 1}(t)\|^2_{\ell^1_2}\leq \|c^{\pm 1}(0)\|^2_{\ell^1_2}+
C\int_0^t\left(\|c^{\pm 1}(s)\|^4_{\ell^1_2}+\|c^{\pm 1}(s)\|^2_{\ell^1_2}\|c^{0}(s)\|^2_{\ell^1_1}\right)ds.
\end{equation}
Moreover we have the following energy inequality:
\begin{eqnarray}\label{energy ineq}
\|c^{\pm 1}(t)||_{\ell^0_2}&+&\left(\frac{\nu+\kappa}{2}\right)\int_0^t\|c^{\pm 1}(s)\|_{\ell^1_2}ds
\nonumber\\
 &\leq&
\|c^{\pm 1}(t)||_{\ell^0_2}+\|c^0(t)\|_{\ell^0_2}+\int_0^t\left(\left(\frac{\nu+\kappa}{2}\right)\|c^{\pm 1}(s)\|_{\ell^1_2}+\nu\|c^0(s)\|_{\ell^1_2}\right)ds
\\
&\leq& \|c^{\pm 1}(0)\|_{\ell^0_2}^2+\|c^0(0)\|_{\ell^1_2}^2.
\nonumber
\end{eqnarray}

In fact, multiply the first equation of \eqref{c} by $c^{0*}(t)$ and the second one by $c^{\pm 1*}(t)$,
we have \eqref{energy ineq}.
Since all  convection terms disappear due to the skew-symmetry, for example,
\begin{eqnarray*}
\sum_n\left(\bar B_n^1(c^1,c^1)\cdot c^{1*}\right)&=&
-i\sum_n\sum_{-n=k+m,\  -\omega_n^1+\omega^1_k+\omega^1_m=0}(q^1_k\cdot m)(q^1_m\cdot q^{-1}_n)c^1_kc^1_mc^{-1}_n\\
&=&
i\sum_n\sum_{-n=k+m,\  -\omega_n^1+\omega^1_k+\omega^1_m=0}(q^1_k\cdot n)(q^{-1*}_{-m}\cdot q^{-1}_n)c^1_kc^{-1*}_{-m}c^{-1}_n\\
&=&
i\sum_m\sum_{m=k+n,\  \omega_n^{-1}+\omega^1_k-\omega^{-1}_m=0}(q^1_k\cdot n)(q^{-1*}_m\cdot q^{-1}_n)c^1_kc^{-1*}_mc^{-1}_n\\
&=&
-\sum_m\left(\bar B_m^{-1}(c^1,c^{-1})\cdot c^{-1*}\right).
\end{eqnarray*}

To apply  Gronwall's inequality, we need the following definition.
Let $C$ be the positive constant appearing in \eqref{apply Gronwall}. Then define $h$ as
\begin{equation*}
h:=\inf\{h'\in [0,\infty):\int _\tau^{\tau+h'}\|c^{\pm
1}(s)\|^2_{\ell_2^1}ds\leq \frac{1}{2C}\quad\text{for any}\quad
\tau>0\}.
\end{equation*}
Note that $h$ is independent of $N$ and can be chosen  positive thaks to
 \eqref{energy ineq}. From \eqref{apply Gronwall} and
an absorbing argument, we see that
\begin{equation*}
\sup_{0<s\leq t}\|c^{\pm 1}(s)\|^2_{\ell^1_2}\leq 2\|c^{\pm
1}(0)\|^2_{\ell^1_2}+ 2C\int_0^t\sup_{0<s''\leq s'}\|c^{\pm
1}(s'')\|^2_{\ell^1_2}\|c^{0}(s')\|^2_{\ell^1_1}ds'\quad\text{for}\quad
t<h.
\end{equation*}
Then by  Gronwall's inequality,
\begin{equation*}
\sup_{0<s\leq t}\|c^{\pm 1}(s)\|_{\ell_2^1}^2\leq 2\|c^{\pm
1}(0)\|_{\ell_2^1}^2\exp\left(2C\int_0^t
\|c^0(s)\|_{\ell_1^1}^2ds\right)\quad\text{for}\quad t<h.
\end{equation*}
Iterating the same argument, one more time, we obtain
\begin{equation*}
\sup_{0<s\leq t}\|c^{\pm 1}(s)\|_{\ell_2^1}^2\leq 2\|c^{\pm
1}(h)\|_{\ell_2^1}^2\exp\left(2C\int_h^t
\|c^0(s)\|_{\ell_1^1}^2ds\right)\quad\text{for}\quad h\leq t<2h.
\end{equation*}

Note that $\int_0^t\|c^0(s)\|^2_{\ell^1_1}ds$ is always finite for
any fixed $t>0$, since there is a global solution to the 2D
Navier-Stokes equations in $\ell_1^s$-type function spaces. Fixing
 $T>0$ and repeating this argument  finitely many times, we have an {\it \`a
priori} bound of $c^{\pm 1}$ in $\ell^1_2$ over $t\in [0,T]$.
Finally, to use Proposition \ref{restconvprop} on the restricted
convolution, we need to verify that  \eqref{restconv} holds. Observe that
\begin{eqnarray*}
P_{n,k,-n-k}(1)&=&
|n|_h^4| k|^4| n+ k|^4+| n|^4| k|_h^4| n+ k|^4
+| n|^4| k|^4| n+ k|^4_h\\
& &
-2| n|^2| n|^2_h| k|^2| k|^2_h| n+ k|^4
-2| n|^2| n|^2_h| k|^4| n+ k|^2| n+ k|_h^2\\
& &
-2| n|^4| k|^2| k|^2_h| n+ k|^2| n+ k|_h^2\\
&=&
| n|^4_hk^8_3+l.o.t.
\end{eqnarray*}
where $l.o.t.$ stands for lower order terms. Thus, it follows that
$P_{n,k,-n-k}(1)$ is a polynomial of degree eight in $k_3$ with a 
nonzero leading coefficient whenever
 $|n_1|+|n_2|\not
=0$. Then for fixed $k_1$, $k_2$ and $n$, there are at most eight
$k_3$ satisfying $\chi(n,k,-n-k)=1$. Thus,
\begin{eqnarray*}
\sum_{2^i\leq |k|\leq 2^{i+1}}|k|^{-1}\chi(n,k,-n-k)&\leq& \sum_{0\leq |k|_h\leq 2^{i+1},k_3\in\mathbb{R}}|k|^{-1}_h\chi(n,k,-n-k)\\
&\leq& 8\sum_{j=1}^i\sum_{2^j\leq |k|_h\leq 2^{j+1}}|k|^{-1}_h\leq 8\sum_{j=1}^i2^{2(j+1)}2^{-j}\leq C2^i.
\end{eqnarray*}

\end{proof}

\section{Proof of the main theorem}

Before proving the main theorem, we first mention the local
existence result.
 Using estimate  (\ref{desiredestimates}), we obtain a local-in-time unique solution to \eqref{original}
in $C([0,T]:\ell^1(\Lambda))$ as stated in the following lemma.

\begin{lemma}\label{local}
Assume that $c(0):=\{c^{\sigma_0}_n(0)\}_{n\in\Lambda,\sigma_0\in\{-1,0,1\}} \in \ell^1(\Lambda)$
and
$c^{\sigma_0}_0(0)=0$ for $\sigma_0\in\{-1,0,1\}$.
  Then there is a local-in-time unique solution
$c(t)\in C([0,T_L]:\ell^1(\Lambda))$
and
$c^{\sigma_0}_0(t)=0$ for $\sigma_0\in\{-1,0,1\}$
satisfying
\begin{equation}\label{T_L}
T_{L}\geq \frac{C}{\|c(0)\|^2},\quad
    \sup_{0<t<T_{L}}\|c(t)\|\leq 10\|c(0)\|,
\end{equation}
where  $C$ is a positive constant independent of $N$.
\end{lemma}
\begin{proof}
First we recall the mild formulation of  \eqref{original}:

\begin{eqnarray*}
c_n^0(t)&=&e^{-\nu |n|^2t}
c^0_n(0)\\
& & +\sum_{(\sigma_1,\sigma_2)\in\{-1,0,1\}^2}
\int_0^te^{-\nu(t-s)|n|^2} \left(\bar
B^0_n(c^{\sigma_1},c^{\sigma_2})+\tilde B^0_n(Nt,
c^{\sigma_1},c^{\sigma_2})\right)ds \end{eqnarray*}
 and
 \begin{eqnarray*}
c_n^{\sigma_0}(t)&=&e^{-\frac{\nu+\kappa}{2} |n|^2t}
c^{\sigma_0}_n(0)\\
& &
+\sum_{(\sigma_1,\sigma_2)\in \{-1,0,1\}^2}
\int_0^te^{-\frac{\nu+\kappa}{2}(t-s)|n|^2}
     \left(\bar B^{\sigma_0}_n(c^{\sigma_1},c^{\sigma_2})+\tilde B^{\sigma_0}_n(Nt, c^{\sigma_1},c^{\sigma_2})    \right)
ds.
\end{eqnarray*}
 By \eqref{desiredestimates}, we have the estimates
\begin{equation*}
\|c_n^0(t)\|\leq \|c^0_n(0)\|+C_\nu
t^{1/2}\sum_{(\sigma_1,\sigma_2)\in\{-1,0,1\}^2}\left(\sup_{0\leq
s<t}\|c^{\sigma_1}(s)\|\sup_{0\leq s<t}\|c^{\sigma_2}(s)\|\right)
\end{equation*}
 and
\begin{eqnarray*}
\|c_n^{\sigma_0}(t)\|&\leq& \| c^{\sigma_0}_n(0)\|\\
 & &+C_{\left(\frac{\nu+\kappa}{2}\right)}t^{1/2}
\sum_{(\sigma_1,\sigma_2)\in \{-1,0,1\}^2}\left(\sup_{0\leq
s<t}\|c^{\sigma_1}(s)\|\sup_{0\leq s<t}\|c^{\sigma_2}(s)\|\right).
\end{eqnarray*}
These {\it \` a-priori} estimates of $\sup_t\|c^0(t)\|$ and
$\sup_t\|c^{\sigma_0}(t)\|$ give us through a standard fixed point
argument the  existence of a local-in-time unique solution (for the
detailed computation, see \cite{GIMM2} for example).
\end{proof}

Let $b^{\sigma_0}(t)$ be the solution to the limit equations \eqref{c} and $c^{\sigma_0}(t)$ be the solution to the original equation \eqref{original}.
The point is to control, in the $\ell^1$-norm, the remainder term $r^{\sigma_0}_n(t):=c^{\sigma_0}_n(t)-b^{\sigma_0}_n(t)$ ($\sigma_0=-1,0,1$) by the large
parameter $N$.
More precisely, $r^0$ and $r^{\sigma_0}$ satisfy
\begin{equation*}
\partial_t r_n^0(t)=-\nu |n|^2r_n^0(t)+\sum_{(\sigma_1,\sigma_2)\in\{-1,0,1\}^2}\left(\bar B^0_n(r^{\sigma_1},c^{\sigma_2})+
\bar B^0_n(b^{\sigma_1},r^{\sigma_2}) +\tilde B^0_n(Nt,
c^{\sigma_1},c^{\sigma_2})\right)
\end{equation*}
and
\begin{eqnarray*}
\partial_t r^{\sigma_0}_n(t)&=&-\left(\frac{\nu+\kappa}{2} \right)|n|^2r^{\sigma_0}_n(t)\\
& &
+\sum_{(\sigma_1,\sigma_2)\in \{-1,0,1\}^2}\left(\bar B^{\sigma_0}_n(r^{\sigma_1},c^{\sigma_2})+
\bar B^{\sigma_0}_n(b^{\sigma_1},r^{\sigma_2})+
\tilde B^{\sigma_0}_n(Nt, c^{\sigma_1},c^{\sigma_2})    \right),
\end{eqnarray*}
respectively.
Once we control the remainder term in the $\ell^1$-norm, we easily
have the main result by a usual bootstrapping argument (see \cite{Y}
for example). Now we show the following  lemma concerning the
smallness of the reminder term. Let
$b(t):=\{b^{\sigma_0}_n(t)\}_{n\in\Lambda,\sigma_0\in\{-1,0,1\}}$
and
$r(t):=\{r^{\sigma_0}_n(t)\}_{n\in\Lambda,\sigma_0\in\{-1,0,1\}}$.

\begin{lemma}\label{fas}
 For all $\epsilon>0$, there is $N_0>0$ such that $\|r_n(t)\|\leq\epsilon$ for $0<t<T_L$ and
 $|N|>N_0$,
 where $T_L$ is the local existence time (see Lemma \ref{local}).
\end{lemma}

\begin{proof}

To simplify the remainder equation, we introduce the following notation.
Let
\begin{eqnarray*}
\bar R^{\sigma_0}_n(r,c,b):&=&\sum_{(\sigma_1,\sigma_2)\in\{-1,0,1\}^2}
\left(\bar B^{\sigma_0}_n(r^{\sigma_1}, c^{\sigma_2})+\bar B^{\sigma_0}_n(b^{\sigma_1}, r^{\sigma_2})
\right),\\
\tilde R^{\sigma_0}_n(Nt,c):&=&\sum_{(\sigma_1,\sigma_2)\in\{-1,0,1\}^2}
\tilde B^{\sigma_0}_n(Nt, c^{\sigma_1}, c^{\sigma_2}).\\
\end{eqnarray*}
We  rewrite the remainder equations as follows:
\begin{equation}\label{remainder equation}
\begin{cases}
\partial_tr_n^0(t)=-\nu |n|^2 r^0_n(t)+\bar R^0_n(r,c,b)+\tilde R^0_n(Nt,c)\\
\partial_t r^{\sigma_0}_n(t)=-\left(\frac{\nu+\kappa}{2}\right) |n|^2r^{\sigma_0}_n(t)
+\bar R^{\sigma_0}_n(r,c,b)+\tilde R^{\sigma_0}_n(Nt,c)
\quad\text{for}\quad \sigma_0=-1,1.\\
\end{cases}
\end{equation}

To control $r$, the key is to  control $\tilde R^{0}_n(Nt,c)$ and $\tilde R^{\sigma_0}_n(Nt,c)$ in \eqref{remainder equation}. To do so,
  we need to analyze the following oscillatory
integral of the non-resonant part as follows:
$$
  \tilde{\mathcal {B}}^{\sigma_0}_n(N t,g^{\sigma_1},h^{\sigma_2}):=
\sum_{n=k+m, \omega^\sigma_{nkm}\not=0}\frac{1}{iN\omega^\sigma_{nkm}}e^{iN t\omega^\sigma_{nkm}}
(q_k^{\sigma_1}\cdot im)(q^{\sigma_2}_m\cdot q_n^{\sigma_0*})
g^{\sigma_1}_{k}h^{\sigma_2}_{m}
$$
and
\begin{equation}\label{estimate R}
\tilde {\mathcal R}^{\sigma_0}_n(Nt,c):=\sum_{(\sigma_1,\sigma_2)\in\{-1,0,1\}^2}
\tilde {\mathcal B}^{\sigma_0}(c^{\sigma_1}, c^{\sigma_2}).
\end{equation}

Note that we have the following relation between $\tilde B$ and $\tilde{ \mathcal B}$ ($\tilde R$ and $\tilde{\mathcal R}$):
\begin{eqnarray*}
  \partial_t\left( \tilde{ \mathcal B}^{\sigma_0}_n(N t,
  g^{\sigma_1},h^{\sigma_2})\right)&=&
   \tilde B^{\sigma_0}_n(N t, g^{\sigma_1},h^{\sigma_2})+ \tilde {\mathcal B}^{\sigma_0}_n(N t,
    \partial_tg^{\sigma_1},h^{\sigma_2})\\
    & & +\tilde{\mathcal B}^{\sigma_0}_n(N t, g^{\sigma_1},\partial_th^{\sigma_2})
\end{eqnarray*}
and
\begin{eqnarray}\label{relation}
  \partial_t\left( \tilde{ \mathcal R}^{\sigma_0}_n(N t,
   c )\right)&=&
   \tilde R^{\sigma_0}_n(N t, c)+ \tilde {\mathcal R}^{\sigma_0}_n(N t,
    \partial_t c).
    \nonumber
\end{eqnarray}

To control $r$, we split \eqref{remainder equation} into two parts: finitely many terms and small (in $\ell^1(\Lambda)$) remainder terms, respectively (cf. \cite[Theorem
6.3]{BMN1}).
 For $\eta=1,2,\cdots$, we choose
$\{s_j\}_{j=1}^\infty\subset\mathbb N$ $(s_1<s_2<\cdots )$ in order
to satisfy $\|(I-\mathcal P_\eta)r\|\to 0\quad (\eta\to\infty)$,
where
\begin{eqnarray*}
  \mathcal P_\eta r&:=&\bigg\{r_{n_1},r_{n_2},\cdots, r_{n_{s_\eta}}: \\
 & &n_1,\cdots,n_{s_\eta}\in\Lambda:
n_k\not=n_\ell\ (k\not=\ell), |n_j|\leq \eta\quad\text{for all}\quad
j=1,\cdots,s_\eta\bigg\}.
\end{eqnarray*}

The choice of $n_1$ $\cdots$ $n_{s_\eta}$ is not uniquely
determined, however this does not matter. Then we can divide $r$ into
two parts: finitely many terms $r_{n_1},\cdots, r_{n_{s_\eta}}$ and
small remainder terms
 $\{(I-\mathcal P_\eta)r_n\}_{n\in\Lambda}$.
\begin{remark}\label{beta}
We have the following estimates:
\begin{equation*}
  \|\mathcal P_\eta \tilde {\mathcal{B}}^{\sigma_0}_n  (\mathcal P_\eta c,\mathcal P_\eta c)\|_0
  \leq \frac{\beta(\eta)}{N}(1+\eta^2)^{1/2}\|\mathcal P_\eta c\|_0^2,
\end{equation*}
\begin{equation*}
\|\mathcal P_\eta \bar {R}^{\sigma_0}_n(\mathcal P_\eta y,c,b)\|
\leq (1+\eta^2)^{1/2}\|\mathcal P_\eta y\|(\|c\|+\|b\|)
\end{equation*}
and
\begin{equation*}
  \| |n|^2\mathcal P_\eta y\|\leq
 (1+\eta^2)\|\mathcal P_\eta y\|_0
\end{equation*}
for $0<t<T_L$ ($T_L$ is a local existence time, see \eqref{T_L}), where
\begin{equation*}
  \beta(\eta):=\max\{|\omega^\sigma_{nkm}|^{-1}:k=k_1,\cdots k_{s_\eta},\ n=n_1,\cdots, n_{s_\eta},\ m=n-k\}.
\end{equation*}
Note that $\beta(\eta)$ is always finite, since it only have finite
combinations for the choice of $n$, $k$ and $m$. We can also have
the same type estimate for $\|\partial_t\mathcal P_\eta c\|$ using
\eqref{c}.
\end{remark}

We use a change of variables to control $\tilde R^0$ and  $\tilde {R}^{\sigma_0}$.
Let us set $y$ as

\begin{equation*}
y^{0}_n(t):= r^{0}_n(t)- \tilde {\mathcal{R}}^{0}_n(N t,\mathcal P_\eta  c)\quad\text{and}\quad
 y^{\sigma_0}_n(t):= r^{\sigma_0}_n(t)- \tilde {\mathcal{R}}^{\sigma_0}_n(N t,\mathcal P_\eta  c).
\end{equation*}
From \eqref{remainder equation}, we see that
\begin{eqnarray*}
\partial_t\left(y^{0}_n+\tilde{\mathcal R}^{0}_n\right)&=&-\nu|n|^2(y^{0}_n+\tilde{\mathcal R}^{0}_n)
+
\bar R^{0}_n(y^{0}+\tilde {\mathcal R}^{0},c,b)\\
& &
+\tilde R^{0}_n(Nt, \mathcal P_\eta c)+\tilde R^{0}_n(Nt, (I-\mathcal P_\eta )c),\\
\partial_t\left(y^{\sigma_0}_n+\tilde{\mathcal R}^{\sigma_0}_n\right)&=&
-\left(\frac{\nu+\kappa}{2}\right)|n|^2(y^{\sigma_0}_n+\tilde{\mathcal
R}^{\sigma_0}_n) + \bar R^{\sigma_0}_n(y^{\sigma_0}+\tilde {\mathcal
R}^{\sigma_0},c,b)\\
& &
+\tilde R^{\sigma_0}_n(Nt, \mathcal P_\eta c) +\tilde
R^{\sigma_0}_n(Nt, (I-\mathcal P_\eta) c).
\end{eqnarray*}
Now we control $\mathcal P_\eta y^0$ and $\mathcal P_\eta y^{\sigma_0}$ for fixed $\eta$. By \eqref{estimate R},
\begin{eqnarray}\label{heat}
\partial_t\mathcal P_\eta y^{0}_n(t)&=&-\nu|n|^2\mathcal P_\eta y^{0}_n
+\mathcal P_\eta \bar { R}^{0}_n(\mathcal P_\eta y,c,b) +E^{0}_n,\\
\partial_t\mathcal P_\eta y^{\sigma_0}_n(t)&=&-\left(\frac{\nu+\kappa}{2}\right)|n|^2\mathcal P_\eta y^{\sigma_0}_n
+\mathcal P_\eta \bar { R}^{\sigma_0}_n(\mathcal P_\eta y,c,b) +E^{\sigma_0}_n,
\nonumber
\end{eqnarray}
where
\begin{eqnarray*}
E^{0}_n:&=&-\mathcal P_\eta \tilde {\mathcal R}^{0}_n(Nt, \mathcal P_\eta \partial_t c)+
\mathcal P_\eta \bar R^{0}_n(\mathcal P_\eta \tilde {\mathcal R}^{0}_n(Nt,\mathcal P_\eta c),c,b)\\
& &-
 \nu|n|^2\mathcal P_\eta \tilde{\mathcal R}^{0}_n(Nt, \mathcal P_\eta c)
 +\mathcal P_\eta\tilde R^0_n(Nt, (I-\mathcal P_\eta)c)\\
 \end{eqnarray*}
 and
 \begin{eqnarray*}
E^{\sigma_0}_n:&=&-\mathcal P_\eta \tilde {\mathcal
R}^{\sigma_0}_n(Nt, \mathcal P_\eta \partial_tc)+
\mathcal P_\eta \bar R^{\sigma_0}_n(\mathcal P_\eta \tilde {\mathcal R}^{\sigma_0}_n(Nt,\mathcal P_\eta c),c,b)\\
& &-
 \left(\frac{\nu+\kappa}{2}\right)|n|^2\mathcal P_\eta \tilde{\mathcal R}^{\sigma_0}_n(Nt, \mathcal P_\eta c)
 +\mathcal P_\eta\tilde R^{\sigma_0}_n(Nt, (I-\mathcal P_\eta)c).
 \end{eqnarray*}
Note that \eqref{heat} are  linear heat type  equations with
external force $E^0$ and $E^{\sigma_0}$. Thus the point is to
control $E^0$ and $E^{\sigma_0}$. By Remark \ref{beta}, we can see
that for any $\epsilon>0$, there is  $\eta_0$  and $N_0$ (depending
on $\eta_0$) such that  if $N>N_0$ and $\eta>\eta_0$, then
$\|E^{\sigma_0}\|<\epsilon$ and $\|E^{\sigma_0}\|<\epsilon$.
Thus we have from \eqref{heat},
\begin{eqnarray*}
\|\mathcal P_\eta y^{0}_n(t)\|&\leq &\int_0^t
\bigg(C_\nu(1+\eta^2)\|\mathcal P_\eta
y^{0}_n(s)\| \\
& &+(1+\eta^2)^{1/2}\|\mathcal P_\eta y^0(s)\|(\|c(s)\|+\|b(s)\|)
+\epsilon\bigg) ds
\end{eqnarray*}
and
\begin{eqnarray*}
 \|\mathcal P_\eta
y^{\sigma_0}_n(t)\|&\leq &\int_0^t \bigg(
C_{\nu,\kappa}(1+\eta^2)\|\mathcal P_\eta
y^{\sigma_0}_n(s)\|\\
& & +(1+\eta^2)^{1/2}\|\mathcal P_\eta
y^{\sigma_0}(s)\|(\|c(s)\|+\|b(s)\|) +\epsilon\bigg) ds.
 \nonumber
\end{eqnarray*}
By Gronwall's inequality, we have that for any $\epsilon>0$, there
is $\eta_0$ and $N_0$ (depending on $\eta_0$) such that if $\eta>\eta_0$
and $N>N_0$, then $\|\mathcal P_\eta y^0\|<\epsilon$ and $\|\mathcal
P_\eta y^{\sigma_0}\|<\epsilon$ for $0<t<T_L$. Clearly, we can also control
$(I-\mathcal P_\eta )y$ with sufficiently large $\eta$ (independent
of $N$), and $\mathcal P_\eta \tilde {\mathcal R}^{\sigma_0}_n(Nt,
\mathcal P_\eta c)$ with sufficiently large $N$ for fixed $\eta$.
Thus we can control $r$ for sufficiently large $\eta$ and $N$.

\end{proof}

{\bf Acknowledgments.}
The second author  thanks  the Pacific Institute for the Mathematical Sciences
 for support of his presence there during the academic
year 2010/2011.
This paper developed during a stay of the second  author as a PostDoc at the Department of Mathematics and Statistics,
University of Victoria.

\end{document}